\documentclass[graybox]{svmult}

\usepackage{type1cm}        % activate if the above 3 fonts are
\usepackage{makeidx}         % allows index generation

\usepackage{hyperref}
\usepackage{graphicx}        % standard LaTeX graphics tool
\usepackage{multicol}        % used for the two-column index
\usepackage[bottom]{footmisc}% places footnotes at page bottom
\usepackage{mathtools} 

\usepackage{newtxtext}       % 
\usepackage[varvw]{newtxmath}       % selects Times Roman as basic font

\usepackage{amsopn}

\usepackage{subfigure}

\usepackage{bm}
\usepackage{enumitem}
\setenumerate{label=\alph*)}
\DeclareMathOperator{\diag}{diag}
\DeclareMathOperator{\Span}{span}
\DeclareMathOperator{\sgn}{sgn}
\newcommand{\N}{\mathbb{N}}
\newcommand{\R}{\mathbb{R}}

\renewcommand{\Re}{\operatorname{Re}}

\renewcommand{\Vec}[1]{\renewcommand*{\arraystretch}{1.2}\begin{pmatrix*}[r]#1\end{pmatrix*}}

\let\epsilon\varepsilon
\let\phi\varphi
\let\rho\varrho

\newcommand{\dt}{\Delta t}

\newcommand{\by}{\mathbf y}
\newcommand{\bT}{\mathbf T}
\newcommand{\bA}{\mathbf A}
\newcommand{\bw}{\mathbf w}

\newcommand{\bg}{\mathbf g}

\newcommand{\bby}{\bar{\by}}
\newcommand{\bS}{\mathbf S}
\newcommand{\bG}{\mathbf G}
\newcommand{\be}{\mathbf e}

\renewcommand{\bf}{\mathbf f}
\newcommand{\bH}{\mathbf H}
\newcommand{\bD}{\mathbf D}
\newcommand{\bxi}{\mathbf \xi}

\newcommand{\bzero}{\bm 0}
\newcommand{\tm}{\subseteq}

\renewcommand{\prod}{p}

\newcommand{\from}{\colon}

\newcommand{\eq}{\text{EQ}}
\newcommand{\gl}{\text{GL}}

\makeindex             % used for the subject index
\begin{document}

\title*{A necessary condition for non oscillatory and positivity preserving time-integration schemes}
\titlerunning{Stability and oscillations of positive schemes}
% Use \titlerunning{Short Title} for an abbreviated version of
% your contribution title if the original one is too long
\author{Th.\ Izgin and P. \"Offner and D. Torlo}
% Use \authorrunning{Short Title for an abbreviated version of
% your contribution title if the original one is too long
\institute{Thomas Izgin \at Department of Mathematics, University of Kassel, Germany, \email{izgin@mathematik.uni-kassel.de}
\and Philipp \"Offfner \at Institute of Mathematics, Johannes Gutenberg University, Mainz, Germany \email{poeffner@uni-mainz.de}
\and Davide Torlo \at mathLab SISSA, SISSA International School of Advanced Studies, Trieste, Italy \email{davide.torlo@sissa.it}}
%
% Use the package "url.sty" to avoid
% problems with special characters
% used in your e-mail or web address
%
\maketitle

%\abstract*{Each chapter should be preceded by an abstract (no more than 200 words) that summarizes the content. The abstract will appear \textit{online} at \url{www.SpringerLink.com} and be available with unrestricted access. This allows unregistered users to read the abstract as a teaser for the complete chapter.
%Please use the 'starred' version of the \texttt{abstract} command for typesetting the text of the online abstracts (cf. source file of this chapter template \texttt{abstract}) and include them with the source files of your manuscript. Use the plain \texttt{abstract} command if the abstract is also to appear in the printed version of the book.}

\abstract{
Modified Patankar (MP) schemes are conservative, linear implicit and unconditionally positivity preserving time-integration schemes constructed for production-destruction systems. 
For such schemes, a classical stability analysis does not yield any information about the performance. Recently, two different techniques have been proposed to investigate the properties of MP schemes. In Izgin \textit{et al.} [ESAIM: M2AN, 56 (2022)], inspired from dynamical systems, the Lyapunov stability properties of such schemes have been investigated, while in Torlo \textit{et al.}[Appl. Numer. Math., 182 (2022)] their oscillatory behaviour has been studied. In this work, we investigate the connection between the oscillatory behaviour and the Lyapunov stability and we prove that a condition on the Lyapunov stability function is necessary to avoid oscillations. We verify our theoretical result on several numerical tests. 
% 120 words
}

%\input{1_introduction}  % label: se:introduction
%% \label se_numeric_method
\section{Introduction}\label{sec:intro}
Consider a production--destruction system (PDS) of ODEs
\begin{equation}\label{eq:PDS}
	\frac{\operatorname{d}y_i(t)}{\operatorname{d}t} = \sum_{j=1}^{I} \left( p_{ij}(y(t))-d_{ij}(y(t))\right), \qquad i=1,\dots,I,\quad t\in \R^{+}_0,
\end{equation}
where $p_{ij},d_{ij}:\R^I\to \R^+_0$ are Lipschitz continuous production and destruction functions, respectively, such that $p_{ij}(y)=d_{ji}(y)$ and $\lim_{y_i\to 0} d_{ij}(y)=0$. Then the system \eqref{eq:PDS} is conservative, i.e., $\sum_i y_i(t) = \sum_i y_i(0)$, and positive, that is, if $y_i(0)\geq 0$ for all $i$, then $y_i(t) \geq 0$ for all $i$.
These systems arise in various fields, e.g. chemical reactions and biological processes, but can be also obtained  from spatial discretisations of hyperbolic conservation/balance laws, e.g. shallow water equations or Euler equations.

Modified Patankar (MP) schemes are conservative, linear implicit and unconditionally positivity preserving time-integration schemes constructed for PDS, inspired by Patankar's original work \cite{patankar1980numerical}.
In recent years, many different MP schemes have been developed \cite{burchard2003high,kopecz2018unconditionally,huang2018third,oeffner_torlo_2019_DeCPatankar}, they have been applied to different applications \cite{ciallella2022arbitrary,huang2019positivity,meister2014unconditionally} and their properties have been studied \cite{kopecz2018order,HIKMSstab22,IKM22Sys,IOE22StabMP,kopecz2019existence,torlo2022issues}.

In the following, we compare the oscillations observed in 2 dimensional systems in \cite{torlo2022issues} and the Lyapunov stability function studied in \cite{izgin2022lyapunov}. Indeed, it is possible to show that a condition on the Lyapunov stability function is necessary to have oscillations--free schemes.
In Section \ref{sec:theorem}, we present the proof of this result; in Section \ref{sec:stability_functions}, we list some stability function of some MP schemes and in Section \ref{sec:numerics} we show how the numerical results validate the theoretical findings.

 %However, MP schemes are in particular constructed for PDS,  classical stability analysis does not yield any information about the performance of such schemes. Recently, two different techniques have been proposed to investigated mP schemes more precisely. In \cite{IKM22Sys,HIKMSstab22,izgin2022lyapunov,IOE22StabMP}, inspired from dynamical systems, the Lyapunov stability properties of such schemes have been investigated where in \cite{torlo2022issues} the oscillatory behaviour have been studied. In the following short note, we investigate the connection between the oscillatory behaviour and the stability properties and prove a necessary conditions on the stability function to avoid oscillations inside the calculations. In several numerical simulations, we verify our theoretical result.
\section{Connection between oscillations and Lyapunov stability}\label{sec:theorem}
We restrict to a linear 2--dimensional problem, in order to have a clear definition of oscillations \cite{torlo2022issues}.
All 2--dimensional linear systems of ODEs that are positive and conservative can be rewritten, with a change of variables, as the following IVP
\begin{equation}\label{eq:Testprob}
	\begin{cases}
	\by'(t)=\bA_\theta\by(t),\\
\by(0)=\by^0>\bzero,
	\end{cases}\quad \bA_\theta=\begin{pmatrix*}[r]
	-\theta &1-\theta\\ \theta & -(1-\theta)
\end{pmatrix*},\quad \theta\in(0,1),
\end{equation}
where this can be seen as PDS, with $p_{12}=d_{21}=(1-\theta)y_2$, $d_{12}=p_{21}=\theta y_1$ and all other entries zero. Let us also consider a one step numerical method whose iterates are generated by a map $\bg$, i.\,e.\ $\by^{n+1} := \bg(\by^n)$. Note that $\bg$ might be given implicitly.

 We first describe oscillations for 2--dimensional linear ODEs through the solution and the steady state. It is known that the exact solution does not overshoot the steady state.
\begin{definition}
	\begin{enumerate}
	\item\label{it:defovershoota} A method is \emph{not overshooting} the steady state of \eqref{eq:Testprob} if $y_2^1<\theta$ and $y_1^1>1-\theta$ for any given initial state $\by^0=(1-\epsilon,\epsilon)^\intercal$ with $\epsilon<\theta$, while when $\epsilon>\theta$ the method is \emph{not overshooting} the steady state if $y_2^1>\theta$ and $y_1^1<1-\theta$.
	\item Otherwise the method is said to be \emph{overshooting} the steady state of \eqref{eq:Testprob}.
	\end{enumerate}	
\end{definition}
\begin{theorem}\label{Thm_MPRK_stabil}
	Let any positive steady state of \eqref{eq:Testprob} be a fixed point of a map $\bg\in \mathcal{C}^2(\R^2_{>0})$. In addition, let the iterates generated by $\by^{n+1}=\bg(\by^n)$ satisfy $\lVert \by^{n+1} \rVert_1=\lVert \by^n \rVert_1$ for all $n\in \N_0$. Finally, let $\by^*$ be the unique positive steady state of \eqref{eq:Testprob}.
	
	Then, the spectrum of the Jacobian $\bD\bg(\by^*)$ is $\sigma(\bD\bg(\by^*))=\{1,R\}$ with $R\in\R$. Furthermore, if $R<0$, then the method generated by $\bg$ is overshooting the steady state of \eqref{eq:Testprob}.
		%\item If $R<0$, then there exist initial values $\by^0$ such that the sequence $(\by^n)_{n\in \N_0}$ oscillates around the stationary solution. \comment{welcher steady state?}
\end{theorem}
%%% Beginn Beweis %%%%%%%%%%%%%%%%%%%%%%%%%
\begin{proof}
	Throughout this proof, we use $\be_1=(1,0)^\intercal$, $\be_2=(0,1)^\intercal$ to denote the standard unit vectors as well as the notation $\bby=(1,-1)^\intercal$. In the proof of \cite[Theorem 2.9]{izgin2022lyapunov}, it is shown that $\bD\bg(\by^*)\by^*=\by^*$ and $\bD\bg(\by^*)\bby=R\bby$ with $R\in \R$, which means that the matrix of eigenvectors
	\begin{equation}\label{eq:Smatrix}
		\bS=(\by^*\;\;  \bby)
	\end{equation}
	is invertible since $\bby$ cannot be a multiple of the positive vector $\by^*$. In particular, we obtain
	\begin{equation*}%\label{eq:DG_digaonal}
		\bS^{-1}\bD\bg(\by^*)\bS=\diag(1,R),
	\end{equation*}
	where $\diag(\by)\in\R^{2\times 2}$ denotes the diagonal matrix with $(\diag(\by))_{ii}=y_i$ for $i=1,2$. Following the lines of the proof of \cite[Theorem 2.9]{izgin2022lyapunov}, we introduce the affine linear transformation $\bT\from\R^2\to\R^2$, \[\by\mapsto\bw=\bT(\by)=\bS^{-1}(\by-\by^*),\]
	where $\bS$ is given in \eqref{eq:Smatrix} and
	the inverse transformation $\bT^{-1}$ is given by \[\bT^{-1}(\bw)=\bS\bw+\by^*.\] 
	
	To see that the method defined by $\bg$ is overshooting $\by^*$, we show that the transformed method given by the map
	\begin{equation*}
		\bG\from \bT(\R_{>0}^2)\to \bT(\R_{>0}^2),\quad \bG(\bw) = \bT(\bg(\bT^{-1}(\bw)))%\label{eq:TgTinv}.
	\end{equation*}
is overshooting the transformed steady state which is $\bw^*=\bzero$.
As demonstrated in \cite[Theorem 2.9]{izgin2022lyapunov}, $\by^0$ is transformed onto the $w_2$-axis and due to the conservation of the map $\bg$, it is proven that $\bG(\bw^0)\in\Span(\bw^0)$ for $\bw^0=(0,w_2^0)^\intercal$. Moreover, 
\begin{equation*}%\label{eq:G}
	\bG(\bw) = \diag(1,R)\bw + \bS^{-1}\bar{\bm R}(\bT^{-1}(\bw))
\end{equation*}
holds, where $\bar{\bm R}$ denotes the Lagrangian remainder
\begin{align}
	(\bar{\bm R}(\by))_i=\frac{1}{2}(\by-\by^*)^\intercal\bH g_i(\by^*+c_i(\by-\by^*))(\by-\by^*),\quad i=1,2\label{Lagrange_remainder}
\end{align}
for some $c_i\in(0,1)$ depending on $\by$ and $\by^*$ and where $\bH g_i$ are the Hessian matrices of $g_i$ for $i=1,2$.
We consider from now on the iterates given by
		\begin{equation*}
		\bw^{n+1}=\begin{pmatrix}
			1 &0\\
			0& R
		\end{pmatrix}\bw^n+\bS^{-1}\bar{\bm R}(\bT^{-1}(\bw^n)),\quad \bw^0=(0,w^0_2)^\intercal.
	\end{equation*}
Here, using $\bS^{-1}=(\widetilde s_{ij})_{i,j=1,2}$ and $w_1^n=0$ it follows from \eqref{Lagrange_remainder} that
\begin{equation}
	(\bS^{-1}\bar{\bm R}(\bT^{-1}(\bw^0)))_1=0
\end{equation}
 since $(\bG(\bw))_1=w_1$. Furthermore,
	\begin{equation}\label{eq:remainder}
		\begin{aligned}
			(\bS^{-1}\bar{\bm R}(\bT^{-1}(\bw^0)))_2=&\frac{1}{2} \sum_{i=1}^2 \widetilde s_{2i}(\bT^{-1}(\bw^0)-\by^*)^\intercal\bH g_i(\bxi^0_i)(\bT^{-1}(\bw^0)-\by^*)\\
			%&+\widetilde s_{22}(\bT^{-1}(\bw^0)-\by^*)^\intercal\bH g_2(\bxi^0)(\bT^{-1}(\bw^0)-\by^*) \Biggr)\\
			=&\frac{1}{2}\sum_{i=1}^2 \widetilde s_{2i}(w_2^0\bS\be_2)^\intercal\bH g_i(\bxi^0_i)(w_2^0\bS\be_2)\\%+\widetilde s_{22}(w_2^0\bS\be_2)^\intercal\bH g_2(\bxi^0)(w_2^0\bS\be_2) \right)\\
			=&\frac{1}{2}\sum_{i=1}^2 \widetilde s_{2i}(w_2^0\bby)^\intercal\bH g_i(\bxi^0_i)(w_2^0\bby)\\%+\widetilde s_{22}(w_2^0\bby)^\intercal\bH g_2(\bxi^0)(w_2^0\bby) \right)\\
			=&C(\bxi^0_1,\bxi^0_2)\cdot (w_2^0)^2,
		\end{aligned}
	\end{equation}
where $\bxi^0_i=\by^*+c_i^0(\by^0-\by^*)$ and $c^0_i\in (0,1)$. Also note that the mapping \mbox{$C\colon \R^2\times \R^2\to\R$} depends on the entries of the Hessians as well as $\bS^{-1}$. 

We now prove that the method defined by $\bG$ is overshooting $\bw^*=\bzero$ by proving the existence of $w_2^0\in \R$ such that $\sgn(w_2^1)\neq\sgn(w_2^0)$. 
We set \[L=\left\{\by\in \R^2\Big| \exists s\in\left[-\tfrac{y_1^*}{2},\tfrac{y_2^*}{2}\right]: \by=\by^*+s\bby\right\}\tm\R^2_{>0}\] and observe that there exists a $K>0$ such that $\sup_{\bxi\in L\times L}\{\lvert C(\bxi_1,\bxi_2)\rvert\}\leq  K<\infty$ since $\bg\in \mathcal C^2$ has bounded second derivatives on the compact set $L$. 

Next, we restrict to $\bw^0$ satisfying $\lvert w_2^0\rvert< \min\left\{\tfrac{y_1^*}{2},\tfrac{y_2^*}{2},\frac{\lvert R\rvert}{K}\right\}$. As a result, $\bw^0=w_2^0\be_2$ yields $\by^0=\bT^{-1}(\bw^0)=\bS\bw^0+\by^*=w_2^0\bby+\by^*\in L$, which means that 
\[\bxi^0_i=\by^*+c^0_i(\by^0-\by^*)=\by^*+c^0_iw_2^0\bby\in L\]
for $i=1,2$.
Now, according to \eqref{eq:remainder}, we have
\begin{equation}
	w^{1}_2=Rw_2^0+C(\bxi^0_1,\bxi^0_2)\cdot(w_2^0)^2=(R+C(\bxi^0_1,\bxi^0_2) w_2^0)w_2^0.\label{w2(n+1)}
\end{equation}
as well as
\begin{equation}
C(\bxi^0_1,\bxi^0_2) w_2^0\leq \lvert C(\bxi^0_1,\bxi^0_2)\rvert \lvert w_2^0\rvert <\lvert C(\bxi^0_1,\bxi^0_2)\rvert\frac{\lvert R\rvert}{K}\leq \lvert R\rvert.\label{cn w2n<R delta}
\end{equation}
Because of $R<0$, the inequality \eqref{cn w2n<R delta} turns into the statement
\begin{equation*}
	R+C(\bxi^0_1,\bxi^0_2)w_2^n<0,
\end{equation*}
and thus, $\sgn(w_2^1)\neq\sgn(w_2^0)$ due to \eqref{w2(n+1)}. This proves that the method defined by $\bG$ is overshooting $\bw^*$ and consequently, the method with iterates given by the map $\bg$ is overshooting $\by^*$.	
\end{proof}
\begin{remark}
It was proven in \cite{izgin2022lyapunov} that if $\lvert R\rvert <1$ holds true, then $\by^*$ is a Lyapunov stable fixed point of the method, whereas it is already well-known that if $\lvert R\rvert >1$ the corresponding fixed point $\by^*$ is unstable, see \cite{SH98} for more details. 
Furthermore, we want to note that for a numerical time-integration method, the eigenvalue $R$ depends on the time step size $\Delta t$, so that $R$ can be interpreted as a \emph{stability function} giving rise to the investigation of stability regions. 
The result from \cite{izgin2022lyapunov} was generalized, see \cite[Theorem 2.9]{IKM22Sys}, and applied to many positivity-preserving schemes in \cite{HIKMSstab22,IKM22Sys, IOE22StabMP}. 
To that end, the corresponding stability functions have been computed, so that we only need to investigate the location of their zeros for investigating the methods with respect to the property of overshooting the steady state of \eqref{eq:Testprob}. 
\end{remark}

\section{Analysis of Modified Patankar Schemes}\label{sec:stability_functions}
In the following, we list the stability functions of some MP schemes. For brevity, we refer to other references for the explicit computations, when available.
%\subsection{Modified Patankar--Runge--Kutta Methods}
As derived in \cite{izgin2022lyapunov}, the stability function of the second order family of MPRK22($\alpha$) schemes, first introduced in \cite{burchard2003high}, is given by
\begin{equation}
	R(z)=\frac{-z^2-2\alpha z+2}{2(1-\alpha z)(1-z)}.
\end{equation}
This function has negative values for negative real part of $z$ if $\Re (z) <- \alpha -\sqrt{\alpha^2+2}$. Hence, for the problem \eqref{eq:Testprob} we obtain the necessary condition 
\begin{equation}\label{eq:lyap_MPRK221}
\dt < \dt_0(\alpha):=\alpha +\sqrt{\alpha^2+2}
\end{equation}
for the method not to overshoot the steady state.
%dt = x + sqrt(x^2 + 2)   $x=alpha$ (in davide's plot to compare)

The stability functions of the families of  MPRK(4,3,$\alpha$, $\beta$) and  MPRK(4,3,$\gamma$) \cite{kopecz2018unconditionally} and the simple MPRK32 \cite{torlo2022issues} are computed in \cite{IOE22StabMP} and not reported here for brevity.

%MPRK(4,3,alpha, beta) complicated 

%MPRK(4,3,gamma)  (stability is indep of gamma)
%\subsection{Strong-Stability-Preserving Modified Patankar--Runge--Kutta Methods}
Similarly, for SSPMPRK schemes we do not report the stability function of SSPMPRK22($\alpha$,$\beta$) \cite{huang2019positivity}, which can be found in \cite{HIKMSstab22}, but we focus on the SSPMPRK43($\eta_2$) for $\eta=\frac{1}{3}$ \cite{huang2018third}.
This scheme possesses the stability function
$
R(z)=\frac{\sum_{i=1}^4a_iz^i}{\sum_{j=1}^4b_jz^j},
$
where, at double precision
\begin{equation*}
	\begin{aligned}
a_0&=1, & b_0&=1,\\
a_1&=-3.349136322977521, & b_1&=-4.349136322977523,\\
a_2&=2.049225690609540, & b_2&= 5.898362013587063,\\
a_3&=0.6815805312568625, & b_3&= -3.208879987508106,\\
a_4&=-0.5093985705698671, & b_4&= 0.6087426554481902.
	\end{aligned}
\end{equation*}
%\subsection{Modified Patankar-Deferred-Correction Methods}

For the Modified Patankar-Deferred-Correction (MPDeC) methods \cite{oeffner_torlo_2019_DeCPatankar}, we derive the stability functions as in \cite{IOE22StabMP} and we show some examples for different orders. The MPDeC schemes are a class of arbitrarily high order positivity preserving methods, based on the Deferred Correction (DeC) methods \cite{Decoriginal,Decremi}. At each stage of the DeC procedure the modified Patankar trick is adopted, carefully choosing the production and destruction terms, according to the DeC coefficients. 
The MPDeC schemes are defined by $M$ subtimesteps and $K$ iterations. The order of accuracy of the MPDeC scheme is the minimum between $K$ and the accuracy of the quadrature formula given by the $M$ subtimesteps. We will focus on equispaced (\eq) and Gauss--Lobatto (\gl) subtimesteps. To obtain order $p$, a number of $K=p$ iterations is required, while we need  $M=\max\{p-1,1\}$ \eq\,subtimesteps or $M=\left\lceil \frac{p}{2} \right \rceil$ \gl\, subtimesteps. The definition of the subtimesteps $0=t^0< \dots < t^M=1$ leads to the definition of the coefficients $\theta_r^m:=\int_{0}^{t^m}\varphi_r(t)dt$ that are the ground component of the MPDeC schemes. Here, $\varphi_r$ is the $r$-th Lagrangian function defined by the subtimenodes $\lbrace t^m \rbrace_{m=0}^M$. 

We denote the MPDeC scheme of order $p$ by MPDeC($p$) and the corresponding stability function $R_p$ can be computed with the following steps:
\begin{equation*}
	\begin{aligned}
			R^{m,(1)}(z)&=\frac{1+2z\sum_{j=0}^M\theta_{j,-}^m}{1-z \sum_{r=0}^M\lvert \theta_r^m\rvert},\\
		R^{m,(\hat k)}(z)&=\frac{1+\theta_0^mz+z\displaystyle\sum_{\substack{j=1\\j\neq m}}^M\theta_j^mR^{j,(\hat k-1)}(z)-z\left( \sum_{\substack{j=0\\j\neq m}}^M\lvert \theta_j^m\rvert -2\theta_{m,-}^m\right)R^{m,(\hat k-1)}(z)}{1-z \sum_{j=0}^M\lvert \theta_j^m\rvert.},\\
		R_p(z)&=R^{M,(K)}(z),
		\end{aligned}
\end{equation*}
for $\hat k=2,\dotsc,K$ and $m=1,\dotsc,M$, where
$
	  \theta_{r,\pm}^m=\frac{\theta_r^m\pm\lvert \theta_r^m\rvert}{2}
$, see \cite{IOE22StabMP} for the details.
We introduce the matrix $\Theta^{X,(p)}\in \R^{M\times (M+1)}$ satisfying 
$\Theta^{X,(p)}_{mr}=\theta_{r-1}^m,
$
where  $X\in\{\eq,\gl\}$ indicates either \eq\, or \gl\, points.
In the case of $p=2$, i.\,e., $M=1$ and $K=2$ we have 
$
		\Theta^{\eq,(2)}=	\Theta^{\gl,(2)}=\Vec{\tfrac12&\tfrac12},
$
that is $\theta_0^1=\theta_1^1=\tfrac12$, and consequently
\begin{equation}
	R_2(z)=\frac{-z^2-2z+2}{2(1-z)^2},
\end{equation}
which equals the stability function of MPRK22($\alpha$) for $\alpha=1$. This is no surprise since MPDeC($2$) is the MPRK22($1$) scheme. 
Next, for $p=3$ we find
\begin{equation*}
	\Theta^{\eq,(3)}=	\Theta^{\gl,(3)}=\Vec{\tfrac{5}{24}& \tfrac13& -\tfrac{1}{24}\\ \tfrac16& \tfrac23& \tfrac16}
\end{equation*}
leading to
\begin{equation*}
R_3(z)=\frac{-331z^5 + 1830z^4 + 3096z^3 - 16452z^2 + 16416z - 5184}{36(-12 + 7z)^2(-1 + z)^3}.
\end{equation*}
Moreover, for $p=4$ and \eq\,subtimesteps we have
\begin{equation*}
	\Theta^{\eq,(4)}=\Vec{\tfrac{1}{8}&\tfrac{19}{72}& -\tfrac{5}{72} &\tfrac{1}{72}\\ \tfrac{1}{9} &\tfrac{4}{9}&\tfrac{1}{9}&0\\
		\tfrac{1}{8} &	\tfrac{3}{8}&	\tfrac{3}{8}&	\tfrac{1}{8}}
\end{equation*}
resulting in
\begin{equation*}
	\begin{aligned}
		R^{\eq}_4(z)& =\frac{\sum_{j=0}^{10}d_jz^j}{1536(-36 + 17z)^3(-3 + 2z)^3(-1 + z)^4},
	\end{aligned}
\end{equation*}
where
\begin{equation*}
	\begin{aligned}
		d_0&= 1934917632,&
		d_1&=-12415721472, &
		d_2&=3402678067,\\
		d_3&= -51295431168,&
		d_4&=45088151040,&
		d_5&= -22031034912,\\
		d_6&= 4329437784,&
		d_7&= 82352116,&
		d_8&=-534268140, \\
		d_9&=64784148, &
		d_{10}&=1805344.
	\end{aligned}
\end{equation*}
On the other hand, for \gl\, and $p=4$ we use $\Theta^{\gl,(4)}=\Theta^{\gl,(3)}$
%\begin{equation*}
%	\Theta^{gl,(4)}=\Vec{\tfrac{\sqrt5+11}{120}&\tfrac{25-\sqrt5}{120}& \tfrac{25-13\sqrt5}{120} &\tfrac{-1+\sqrt5}{120}\\ \tfrac{11-\sqrt5}{120} &\tfrac{25+13\sqrt5}{120}&\tfrac{25+\sqrt5}{120}&-\tfrac{1+\sqrt5}{120}\\
%		\tfrac{1}{12} &	\tfrac{5}{12}&	\tfrac{5}{12}&	\tfrac{1}{12}}
%\end{equation*}
  with $K=4$ and $M=2$, obtaining
a rational function with a polynomial of degree $7$ in the numerator and denominator, which can be represented by
\begin{equation*}
	\begin{aligned}
	R^{\gl}_4(z)& =\frac{\sum_{j=0}^{7}c_jz^j}{(7z\sqrt5 + 5z - 60)^3(7z\sqrt5 + 31z - 60)^4},
\end{aligned}
\end{equation*}
where
\begin{equation*}
	\begin{aligned}
	 c_0&= -279936\cdot 10^7,&
	c_1&=(1982880\sqrt5 + 5062176)\cdot 10^6, \\
	c_2&=(-28409616\sqrt5 - 58953744)\cdot 10^5,&
	c_3&= (157481496\sqrt5 + 347034456)\cdot 10^4,\\
	c_4&= -262068264000\sqrt5 - 617156712000,&
	c_5&= -55771610400\sqrt5 - 129811572000,\\
	c_6&=13763385600\sqrt5 + 34116840000,&
	c_7&= 1038579760\sqrt5 + 2083625200.
\end{aligned}
\end{equation*}
For higher order and other schemes, we refer to the Maple code in the reproducibility repository \cite{ourrepo}.
\section{Numerical Comparison}\label{sec:numerics}
In this section, we compare the numerical bound $\dt_0$ for $\dt$ not to be oscillating \cite{torlo2022issues} with the necessary condition given by the Lyapunov stability function derived following \cite{izgin2022lyapunov}. The Julia Jupyter notebook used to compute the numerical bound and the Maple notebook where the Lyapunov stability functions are computed are available in the reproducibility repository \cite{ourrepo}.
Those notebooks can be used also to compute the bounds for different parameters of the presented schemes that could not fit in this work.
\begin{figure}
	\subfigure[MPRK22($\alpha$)\label{fig:MPRK22alpha}]{
		\includegraphics[width=0.49\textwidth]{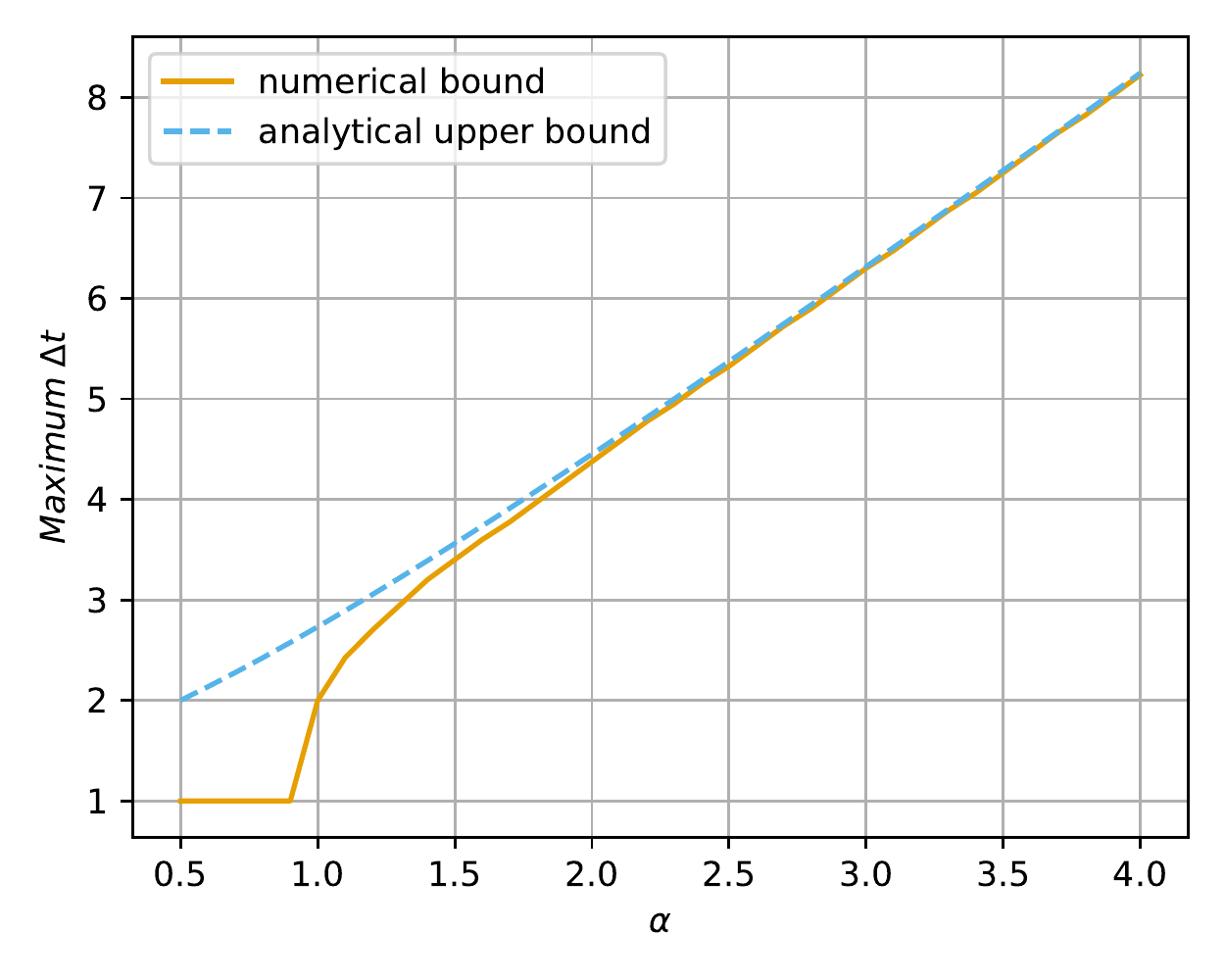}
	}\hfill
	\subfigure[MPRK43($\gamma$), analytical upper bound 2.35\label{fig:MPRK43gamma}]{
		\includegraphics[width=0.49\textwidth]{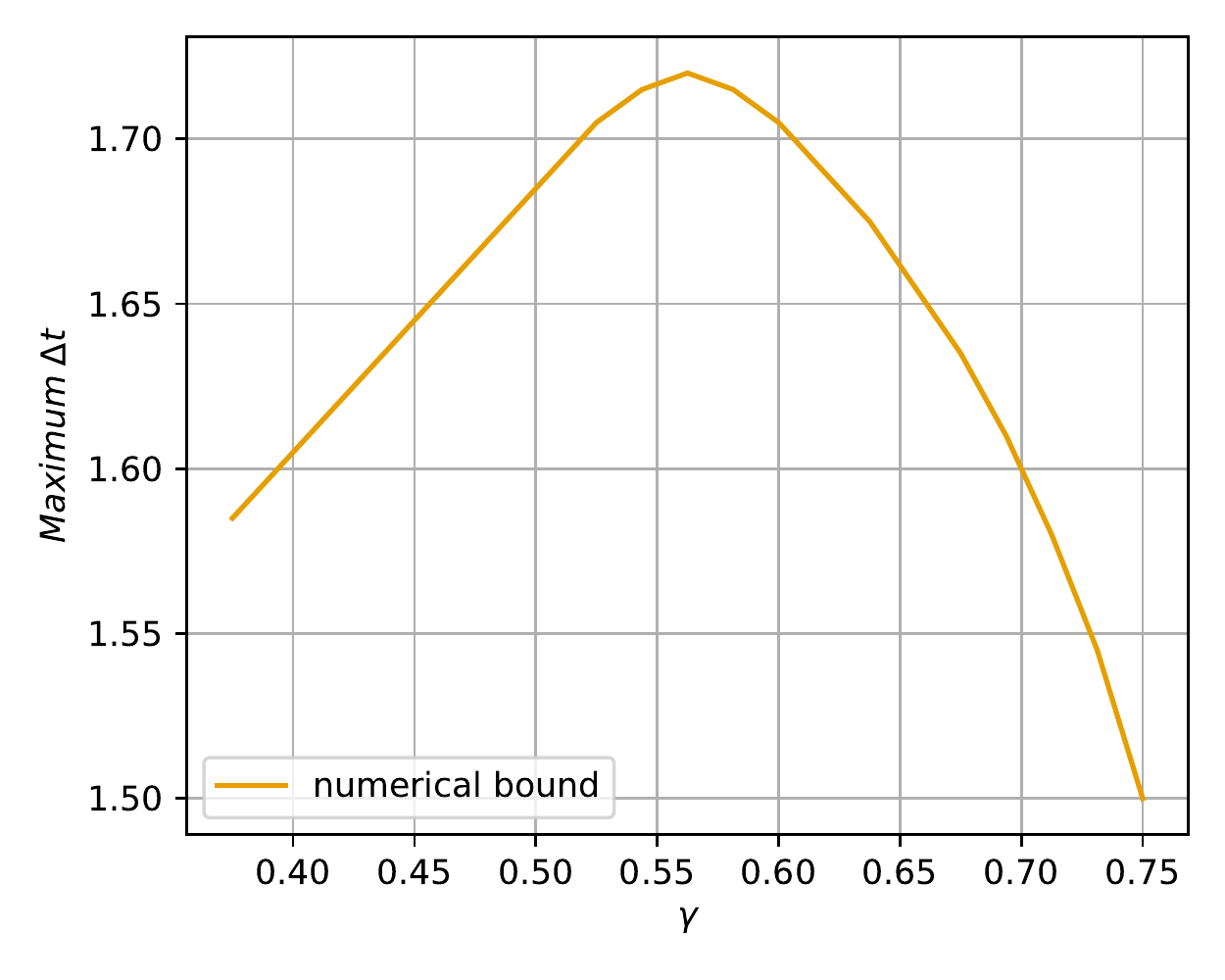}
	}
\caption{Plot of numerical bound for $\dt$ (orange) and Lyapunov stability $\dt$ bound \eqref{eq:lyap_MPRK221} (blue) for the MPRK22($\alpha$) and MPRK43($\gamma$) families of schemes}\label{fig:mprk}
\end{figure}

In Figure \ref{fig:MPRK22alpha}, we show the two bounds on $\Delta t$ for MPRK(2,2,$\alpha$) \cite{burchard2003high} varying $\alpha$. 
We observe that there is a very good agreement between the two conditions for $\alpha>1.5$, while for smaller values the error is bounded by $\sqrt{3}$.
For the MPRK(4,3,$\gamma$) \cite{kopecz2019existence} we observe that the numerical bound in Figure \ref{fig:MPRK43gamma} is not as close as before to the Lyapunov stability bound $2.35$ (independently on $\gamma$), but still it is giving an indication of the magnitude of the bound.

In Tables \ref{tab:mpdecequi} and \ref{tab:mpdecGLB}, we write the numerical $\dt$ bound and the necessary condition given by the Lyapunov stability function in Theorem \ref{Thm_MPRK_stabil} for \eq\,and \gl\, MPDeC, respectively. Here, we notice very different behaviors between \eq\, and \gl\, MPDeC. In the \eq\, case, the bounds are widely varying across different orders of accuracy, in the numerical simulations, while for the theoretical bound, we get very large constraints that are not very useful. On the other side, for \gl, the numerical bounds converge very quickly to 1 as the order increases. The Lyapunov stability function leads to a not so sharp bound, but much closer to the numerical one.
\begin{figure}
	\subfigure[MPDeC \eq]{
	\begin{tabular}{|c|c|c|}\hline
		$p$ & num. $\dt_0$ & Lyap. $\dt_0$ \\ \hline 
		1 & $\infty$ &  $\infty$\\
		2 & 2.0 & $2.73$\\ 
		3 & 1.19& 3.31\\ 
		4 & 1.11& 3.83\\ 
		5 & 1.07& 4.19\\ 
		6 & 1.04& $\infty$\\ 
		7 & 1.04& $\infty$\\ 
		8 & 1.37& $\infty$\\ 
		9 & 6.96& $\infty$\\ 
%		10 & 1.0&3.66\\ 
%		11  &15.5&\\ 
%		12  &1.0 &\\ 
%		13  &35.51& \\ 
%		14  &1.07 &\\ 
%		15  &12.13&\\ 
%		16  &1.80 &\\ 
\hline
	\end{tabular}\label{tab:mpdecequi}
}
\subfigure[MPDeC \gl]{
	\begin{tabular}{|c|c|c|}\hline
		$p$ & num. $\dt_0$ & Lyap. $\dt_0$ \\ \hline 
		1 & $\infty$ &  $\infty$ \\
		2 & 2.0 & $2.73$\\ 
		3 & 1.19 & 3.31\\ 
		4 & 1.07 & 3.62\\ 
		5 & 1.04 & 3.74\\ 
		6 & 1.0 & 4.06\\ 
		7 & 1.0 & 4.47\\ 
		8 & 1.0 & 5.03\\ 
		9 & 1.0 & 20.1\\ 
	%	10 & 1.0 &\\ 
%		11 & 1.0 &\\ %
%		12 & 1.0 &\\ 
%		13 & 1.0 &\\ 
%		14 & 1.0 &\\ 
%		15 & 1.0 &\\ 
%		16 & 1.0 &\\ 
 \hline
	\end{tabular}\label{tab:mpdecGLB}
}\hfill
\subfigure[Other schemes]{
	\begin{tabular}{|c|c|c|}\hline
		Method & num. $\dt_0$ & Lyap. $\dt_0$ \\ \hline 
		SSPMPRK(4,3) & 1.31 & 2.15\\
		MPRK(3,2) & 16.56 & $\infty$\\ 
		MPRK(4,3,2,0.6) & 1.89 & 3.07\\ 
		MPRK(4,3,0.9,0.5) & 1.59 & 2.82\\ 
		MPRK(4,3,0.5,0.7) & 1.74 & 2.00\\ 
		MPRK(4,3,3,$\tfrac{7}{15}$) & 5.37 &5.62\\ 
		SSPMPRK(2,2,0,1) & 2 & $2.73$\\ 
		SSPMPRK(2,2,0,2) & 4.36 & $4.45$\\ 
		SSPMPRK(2,2,0.4,1) & 1.27 & 2.14\\ 
		SSPMPRK(2,2,0.1,4) & 2.10 &2.37\\ 
		\hline
	\end{tabular}\label{tab:otherschemes}
}
\caption{Numerical bound for $\dt$ and Lyapunov stability function $\dt$ bound  for various schemes}\label{fig:mpdecother}
\end{figure}

In Table \ref{tab:otherschemes}, we summarize the results for a selection of other schemes for various parameters.
In all cases, we observe, as predicted by Theorem \ref{Thm_MPRK_stabil}, that the numerical bound is smaller than the Lyapunov stability function bound.
The discrepancy between the two approaches vary a lot between different schemes and even between different parameters of the same method family, as already observed for MPRK(2,2,$\alpha$).
We observe, in general, lower discrepancy for second order schemes, e.g. SSPMPRK(2,2,0,2) and SSPMPRK(2,2,0.1,4), and higher discrepancy for higher order schemes, e.g. SSPMPRK(4,3) and MPRK(4,3,2,0.6). A special remark on MPRK(3,2) is necessary, as it is the second order scheme with the largest $\dt_0$. Its numerical bound is very large $\approx 16.5$, while there is no Lyapunov stability function bound. This shows, again, that this scheme performs very robustly in these simulations. % label se_numerics
\section{Conclusion}\label{sec:conclusion}
We have shown that the oscillations that modified Patankar schemes show in two--dimensional systems are linked to the Lyapunov stability function. 
In particular, it is necessary that the Lyapunov stability function is nonnegative to have an oscillations--free method. 
In particular, these conditions are verified for $\dt\leq \dt_0$, where $\dt_0$ depends on the scheme. 
We validated the theoretical results with many numerical tests showing that the bound coming from the Lyapunov stability function is always larger than the numerical one.

The found results are useful to choose the time step to avoid oscillations. In many situations, the theoretical bound and the numerical one are actually very close and this gives an indication on how to adopt the time step. 
Furthermore, there are still open questions on the behavior of MP schemes, in particular for hyperbolic problems, where the positivity of various physical quantities is of paramount importance. 
We plan to extend this work to a stability analysis of fully discrete MP schemes hoping to find connections with the found oscillations bounds. Furthermore, it is of interest to investigate Lyapunov stability properties in the context of partial differential equations.

%Furthermore, the authors believe that there is a connection between these type of oscillations and oscillations that typically arise in hyperbolic problems close to discontinuities. 
%Hence, we plan to study in the future the  TVD  property of MP time discretization for hyperbolic problems.  ??????????????????????????????????????????????? % label se_conclusion

\section*{Acknowledgements}
The author Th. Izgin gratefully acknowledges the financial support by the Deutsche Forschungsgemeinschaft (DFG) through grant ME 1889/10-1. P. \"Offner was supported by the Gutenberg Research College, JGU Mainz. D. Torlo (Sissa, Italy) was supported by a SISSA Mathematical Fellowship.

\bibliographystyle{siam}
\bibliography{literature}

\end{document}